\documentclass[12pt,leqno]{amsart}
\usepackage{amsmath,amsthm,amsfonts,amssymb, amscd,amsbsy,multirow, hyperref}
\usepackage{hyperref}
\usepackage{graphicx}
\usepackage{xcolor}
\graphicspath{ }
\usepackage{wrapfig}
\usepackage{accents}
\usepackage{caption}
\usepackage{subcaption}
\usepackage{calligra}

\numberwithin{figure}{section}

\theoremstyle{plain}
\newtheorem{thm}{Theorem}[section]
\newtheorem{lem}[thm]{Lemma}

\newtheorem{cor}{Corollary}[thm]
\theoremstyle{definition}
\newtheorem{defn}{Definition}[section]

\theoremstyle{remark}

\usepackage{mathtools}
\title[On gradient $\rho$-Einstein solitons]{Diameter estimation of gradient $\rho$-Einstein solitons}
\author[A. A. Shaikh, P. Mandal and C. K. Mondal]{Absos Ali Shaikh$^{*1}$, Prosenjit Mandal$^{2}$ and Chandan Kumar Mondal $^{3\dag}$}
\address{$^{1,2,3}$Department of Mathematics,\newline University of
Burdwan, Golapbag,\newline Burdwan-713104,\newline West Bengal, India.}

\address{$^{\dag}$Department of Mathematics,\newline School of Sciences,\newline Durgapur Regional Center\newline Netaji Subhas Open University,\newline Durgapur-713214,\newline West Bengal, India.}
\email{$^1$aask2003@yahoo.co.in, aashaikh@math.buruniv.ac.in}
\email{$^2$prosenjitmandal235@gmail.com}
\email{$^3$chan.alge@gmail.com, chandanmondal@wbnsou.ac.in}
\begin{document}
\begin{abstract}
Our aim in this article is to give a lower bound of the diameter of a compact gradient $\rho$-Einstein soliton satisfying some given conditions. We have also deduced some conditions of the gradient $\rho$-Einstein soliton with bounded Ricci curvature to become non-shrinking and non-expanding. Further, we have proved that a complete non-compact gradient shrinking or expanding Schouten soliton with non-constant potential and a boundedness condition on scalar curvature must be non-parabolic. 
\end{abstract}
\noindent\footnotetext{
$*$ Corresponding author.\\
$\mathbf{2010}$\hspace{5pt}Mathematics\; Subject\; Classification: 53C20; 53C21; 53C25.} 
\maketitle
\section{Introduction and preliminaries}
The Ricci flow introduced by Hamilton \cite{HA82} plays a significant role in Perelman's proof of the Poincar\'e conjecture and currently it has been intensively used in the study of various geometric properties.
For study on Ricci flow, see \cite{Chow-B}.

An important aspect in the investigation of the Ricci flow is the study of Ricci solitons. A gradient Ricci soliton is an $n(\geq 2)$-dimensional Riemannian manifold $(M, g)$ with Riemannian metric $g$, satisfying
\begin{equation}\label{rs1}
 Ric + \nabla^2 f = \lambda g,
 \end{equation}
where $\nabla^2f$ stands for the Hessian of $f\in C^{\infty}(M)$, the ring of smooth functions on $M$, $Ric$ is the Ricci curvature tensor and $\lambda \in \mathbb{R}$. A Ricci soliton $(M, g)$ is called expanding if $\lambda <0$, steady if $\lambda=0$ and shrinking if $\lambda > 0$. For some results of Ricci solitons see \cite{CA2019, CA2020, AC2021}. In general, it is natural to consider geometric flows of the following type on a $n(\geq 3)$- dimensional Riemannian manifold $(M,g)$:
$$\frac{\partial g}{\partial t}=-2({\rm Ric}-\rho Rg),$$
where $R$ denotes the scalar curvature of the metric $g$ and $\rho\in\mathbb{R}\diagdown\{0\}$. The parabolic theory for these flows was developed by Catino et. al. \cite{Catino13}, which was first considered by Bourguignon \cite{Bour}. They called such a flow as Ricci-Bourguignon flows. they defined the following notion of
 $\rho$-Einstein solitons.
 
 \begin{defn}
Let $(M,g)$ be a Riemannian manifold of dimension $n(\geq 3)$, and let $\rho\in\mathbb{R}$, $\rho\neq 0$. Then $M$ is called a $\rho$-Einstein soliton if there is a smooth vector field $X$ such that

\begin{equation}\label{a2}
{\rm Ric}+\frac{1}{2}\mathcal{L}_Xg-\rho Rg=\lambda g,
\end{equation}
where $\mathcal{L}_Xg$ represents the Lie derivative of g in the direction of the vector field $X$.  
\end{defn}
If there exists a smooth function $f:M\rightarrow\mathbb{R}$ such that $X=\nabla f$ then the $\rho$-Einstein soliton is called a gradient $\rho$-Einstein soliton, denoted by $(M,g,f)$ and in this case (\ref{a2}) takes the form 
 
\begin{equation}\label{res1}
 Ric+\nabla^2 f-\rho Rg=\lambda g.
\end{equation}
The function $f$ is called a $\rho$-Einstein potential of the gradient $\rho$-Einstein soliton. As usual, a $\rho$-Einstein soliton is called steady for $\lambda=0$, shrinking for $\lambda>0$ and expanding for $\lambda<0$. After rescaling the metric $g$ we may assume that $\lambda \in \{-\frac{1}{2}, 0,\frac{1}{2}\}$. For more study on $\rho$-Einstein solitons, we refer to the interested reader \cite{CAT16,  AAP2021, Absos} and also the references therein.

For particular value of the parameter $\rho$, a $\rho$-Einstein soliton is called 
\begin{enumerate}
\item[i)] gradient Einstein soliton if $\rho=\frac{1}{2}$,
\item[ii)] gradient traceless Ricci soliton if $\rho=\frac{1}{n}$,
\item[iii)] gradient Schouten soliton if $\rho=\frac{1}{2(n-1)}$.
\end{enumerate}
Taking trace of $(\ref{res1})$ we obtain
\begin{equation}\label{res2}
 R+\Delta f-n\rho R=n\lambda.
 \end{equation}
 Thus for gradient Schouten soliton, $(\ref{res2})$ takes the form
 \begin{equation}\label{ss1}
  R+\Delta f-\frac{nR}{2(n-1)}=n\lambda.
 \end{equation}
 The diameter estimation of Ricci soliton is an abuzz topic of research. One of the first result came from the work of Myers \cite{MY41}. In particular, he showed that a complete $n$-dimensional Riemannian manifold $(M,g)$ with Ricci curvature satisfying $Ric\geq \lambda g$ for some positive constant $\lambda$ is compact with the diameter $diam(M)$ having an upper bound $\pi \sqrt{(n-1)/\lambda}$. After that many authors have investigated to find a bound for a manifold satisfying some curvature flow conditions, for example see \cite{FG10, FS13, LW14}. For the complete literature in this topic see the survey article \cite{TA19}. 
 \par The paper is organized as follows: In the section 1, we have deduced a lower bound of the gradient $\rho$-Einstein soliton satisfying some curvature conditions. We have also found some conditions of the gradient $\rho$-Einstein soliton for being non-shrinking and non-expanding. Finally, section 2 deals with the non-parabolic behavior of Schouten soliton. 

\section{Diameter estimation}
\begin{thm}\label{resth1}
 Let $(M,g,f)$ be a complete non-compact gradient $\rho$-Einstein soliton with bounded Ricci curvature, i.e., $|Ric|\leq c$ for some constant $c$, $\rho R \geq c_1\lambda$ for some constant $c_1$ and $\lim\limits_{s_0\rightarrow \infty}\int_{0}^{s_0}\nabla^2 f(X,X)$ is finite. Then the $\rho$-Einstein soliton is non-shrinking if $(1+c_1)>0$ and non-expanding if $(1+c_1)<0$.
 \end{thm}
\begin{proof}
Let us consider a length minimizing normal geodesic $\gamma: [0,s_0]\rightarrow M$ for some positive, arbitrarily large $s_0$. Take $p=\gamma(0)$ and $X(s)=\gamma'(s)$ for $s>0$. Then $X$ is the unit tangent vector along $\gamma$. Now integrating $(\ref{res1})$ along $\gamma$, we get
\begin{eqnarray}\label{r6}
\int_{0}^{s_0}Ric(X,X)&=& \int_{0}^{s_0}(\lambda+\rho R) g(X,X)-\int_{0}^{s_0}\nabla^2 f(X,X)\nonumber \\
&\geq& \lambda(1+c_1) s_0-\int_{0}^{s_0}\nabla^2 f(X,X).
\end{eqnarray}
Again, the second variation of arc length implies that
\begin{equation}
\int_{0}^{s_0}\psi^2Ric(X,X)\leq (n-1)\int_{0}^{s_0}|\psi'(s)|^2 ds,
\end{equation}
for every non-negative function $\psi$ defined on $[0,s_0]$ with $\psi(0)=\psi(s_0)=0$. We now choose the function $\psi$ as follows:
\[\psi(s)= \begin{cases} 
      s & s\in[0,1] \\
      1 & s\in[1,s_0-1] \\
      s_0-s & s\in[s_0-1,s_0].
   \end{cases}
\]
Then, we have
\begin{eqnarray}\label{r2} 2(n-1)+\sup_{B(p,1)}|Ric|+\sup_{B(\gamma(s_0),1)}|Ric|&\geq & (n-1)\int_{0}^{s_0}|\psi'(s)|^2 ds+\int_{0}^{s_0}(1-\psi^2)Ric(X,X)ds\nonumber\\
&\geq & \int_{0}^{s_0}\psi^2 Ric(X,X)ds+\int_{0}^{s_0}(1-\psi^2)Ric(X,X)ds\nonumber\\
&=&\int_{0}^{s_0}Ric(X,X)ds .
\end{eqnarray}
Combining (\ref{r6}) and (\ref{r2}), we obtain
\begin{eqnarray}\label{r3}
\lambda(1+c_1) s_0-\int_{0}^{s_0}\nabla^2 f(X,X)&\leq & 2(n-1)+\sup_{B(p,1)}|Ric|+\sup_{B(\gamma(s_0),1)}|Ric|\nonumber \\
&\leq& 2(n-1)+2 c.
\end{eqnarray}
Therefore, taking limit as $s_0\rightarrow \infty$ on both sides of (\ref{r3}), we can write
\begin{equation}\label{r4}
\lim\limits_{s_0\rightarrow\infty}\lambda(1+c_1) s_0-\lim\limits_{s_0\rightarrow\infty}\int_{0}^{s_0}\nabla^2 f(X,X)\leq 2(n-1)+2 c.
\end{equation}
Now since $\lim\limits_{s_0\rightarrow\infty}\int_{0}^{s_0}\nabla^2 f(X,X)$ is finite, hence, if $(1+c_1)>0$ and $\lambda>0$, then $\lim\limits_{s_0\rightarrow\infty}\lambda(1+c_1) s_0=+\infty$, which contradicts the inequality (\ref{r4}). Thus $\lambda\leq 0$, i.e., the $\rho$-Einstein soliton is non-shrinking. In a similar way we can show that if $(1+c_1)<0$ then $\lambda\geq 0$, i.e., the $\rho$-Einstein soliton is non-expanding.  
\end{proof}
 We know that for a non-trivial concave function $f\in C^{\infty}(M)$, the function $(-f)$ is non-constant convex, also it implies that $M$ is non-compact and $\lim\limits_{s_0\rightarrow\infty}\int_{0}^{s_0}\nabla^2 f(X,X)\leq 0$. Thus from the above Theorem \ref{resth1} we can write the following corollary:
 \begin{cor}
 Let $(M,g,f)$ be a complete gradient $\rho$-Einstein soliton with bounded Ricci curvature, i.e., $|Ric|\leq c$ for some constant $c$, $\rho R \geq c_1\lambda$ for some constant $c_1$ and $f$ is a non-constant concave function. Then the $\rho$-Einstein soliton is non-shrinking if $(1+c_1)>0$ and non-expanding if $(1+c_1)<0$.
 \end{cor}

\begin{thm}\label{th3}
 Let $(M,g,f)$ be a compact gradient $\rho$-Einstein soliton with $c_2 g\leq Ric \leq c_3 g$. Then for $\rho>0$,
 $$ diam(M)\geq max \Big\{\sqrt{\frac{2(f_{max}-f_{min})}{\lambda+n\rho c_3-c_2}}, \sqrt{\frac{2(f_{max}-f_{min})}{c_3-\lambda-n\rho c_2}}, \sqrt{\frac{8(f_{max}-f_{min})}{(n\rho+1)(c_3-c_2)}}\Big\},$$ 
 and for $\rho<0,$
  $$ diam(M)\geq max \Big\{\sqrt{\frac{2(f_{max}-f_{min})}{\lambda+n\rho c_2-c_2}}, \sqrt{\frac{2(f_{max}-f_{min})}{c_3-\lambda-n\rho c_3}}, \sqrt{\frac{8(f_{max}-f_{min})}{(n\rho-1)(c_2-c_3)}}\Big\},$$
 where the numbers $c_2$, $c_3$ are denoted by
 \begin{eqnarray*}
 c_2=inf_{x\in M}\{Ric(v,v):v\in T_x M, g(v,v)=1\},\\
  c_3=sup_{x\in M}\{Ric(v,v):v\in T_x M, g(v,v)=1\}.
 \end{eqnarray*}
 \end{thm}
\begin{proof}
Taking trace of $c_2 g\leq Ric \leq c_3 g$, we obtain
\begin{equation}\label{eq1}
nc_2\leq R \leq n c_3.
\end{equation}
As $\rho>0$, the above inequality yields
\begin{equation}\label{e7}
n\rho c_2\leq \rho R \leq n\rho c_3.
\end{equation} 
The potential function $f$ has at least one point $p$ where it attains its global minimum value, as $M$ is compact. Let $\gamma$ be a geodesic with $\gamma
 (0)=p$. Then using $(\ref{res1})$ and $(\ref{e7})$ we calculate 
 \begin{eqnarray}\label{e1}
\nonumber g(\nabla f,\gamma')(\gamma (s))&=& g(\nabla f,\gamma')(\gamma (s))-g(\nabla f,\gamma')(\gamma (0))\\
\nonumber &=& \int_{0}^{s} \frac{\partial}{\partial s}g(\nabla f,\gamma')(\gamma (s)) ds\\
\nonumber &=&\int_{0}^{s} \nabla _{\gamma'}g(\nabla f,\gamma')(\gamma (s)) ds\\ \nonumber&=&\int_{0}^{s} \nabla^2 f(\gamma',\gamma')(\gamma (s)) ds\\
&\leq&(\lambda+n\rho c_3-c_2 )s.
 \end{eqnarray}
 Integrating $(\ref{e1})$ we get
$$ f(\gamma(s))-f(p)\leq \frac{(\lambda+n\rho c_3-c_2 )s^2}{2}.$$
Since for every point $x\in M$ there exists a minimizing geodesic joining $p$ and $x$, for all $x\in M$ we have 
\begin{equation}\label{e3}
f(x)-f(p)\leq \Big( \frac{\lambda+n\rho c_3-c_2}{2}\Big)d^2(x,p),
\end{equation}
where $d(x,p)$ is the distance between $x$ and $p$.\\
 In particular, we obtain
$$f_{max}-f_{min}\leq \Big( \frac{\lambda+n\rho c_3-c_2}{2}\Big)d^2,$$
where $d=diam(M)$, is the diameter of the manifold $M$. This gives
$$d^2\geq \Big(\frac{2(f_{max}-f_{min})}{\lambda+n\rho c_3-c_2}\Big).$$
Now we consider a point $q$ at which $f$ attains its global maximum. Let $\gamma$ be a geodesic with $\gamma(0)=q$. Then
\begin{eqnarray}\label{e2}
\nonumber g(\nabla f,\gamma')(\gamma (s))&=& g(\nabla f,\gamma')(\gamma (s))-g(\nabla f,\gamma')(\gamma (0))\\
\nonumber &=& \int_{0}^{s} \frac{\partial}{\partial s}g(\nabla f,\gamma')(\gamma (s)) ds\\
\nonumber &=&\int_{0}^{s} \nabla _{\gamma'}g(\nabla f,\gamma')(\gamma (s)) ds\\ \nonumber&=&\int_{0}^{s} \nabla^2 f(\gamma',\gamma')(\gamma (s)) ds\\
&\geq&\Big(\lambda+n\rho c_2-c_3\Big)s.
 \end{eqnarray}
 Again integrating $(\ref{e2})$ we get
 $$ f(\gamma(s))-f(q)\geq \frac{(\lambda+n\rho c_2-c_3)}{2}s^2.$$
 Since for every point $x\in M$ there exists a minimizing geodesic joining $q$ and $x$, for all $x\in M$ we have 
 $$f(x)-f(q)\geq \Big(\frac{\lambda+n\rho c_2-c_3}{2}\Big)d^2(x,q),$$
 where $d(x,q)$ is the distance between $x$ and $q$.\\
 This implies that
 \begin{equation}\label{e4}
 f(q)-f(x)\leq \Big(\frac{c_3-\lambda-n\rho c_2}{2}\Big)d^2(x,q).
 \end{equation}
  In particular, we obtain
 $$f_{max}-f_{min}\leq \Big(\frac{c_3-\lambda-n\rho c_2}{2}\Big)d^2,$$
which yields
 $$d^2\geq \Big(\frac{2(f_{max}-f_{min})}{c_3-\lambda-n\rho c_2}\Big).$$
 Finally, adding $(\ref{e3})$ and $(\ref{e4})$ for $x$ such that $d(x,p)=d(x,q)\leq\frac{d}{2}$, we get
 \begin{eqnarray*}
 f(q)-f(p)&\leq& \Big(\frac{c_3-\lambda-n\rho c_2}{2}\Big)d^2(x,q)+\Big( \frac{\lambda+n\rho c_3-c_2}{2}\Big)d^2(x,p)\\
 &\leq& \frac{(n\rho+1)(c_3-c_2)}{8}d^2.
 \end{eqnarray*}
 This implies
 $$d^2\geq \frac{8(f_{max}-f_{min})}{(n\rho+1)(c_3-c_2)}.$$
 This proves the first part. For the second part, $\rho<0$, the equation (\ref{eq1}) implies that
 \begin{equation}\label{e5}
 n\rho c_3\leq \rho R \leq n\rho c_2,
 \end{equation} 
 and hence proceeding in a similar way as in the first case, we obtain the second part.
 \end{proof}
\section{Schouten solitons}
A Riemannian manifold $M$ is parabolic if every subharmonic function $u$ on $M$ with $u^*=sup_M u <\infty$, must be constant \cite{AG1999,GX2019}, equivalently, if every positive superharmonic function $u$ on $M$ is constant. Otherwise $M$ is said to be non-parabolic. The Green function $G(x,y)$ on $M$ is defined by (see, \cite{AG1999})
$$G(x,y)=\frac{1}{2}\int_{0}^{\infty}k(t,x,y)dt,$$
where $k(t,x,y)$ is the heat kernel of $M$.
If $p$ is a fixed point on $M$ and $M$ is non-parabolic then there is a unique, minimal, positive Green function and is denoted by $G(p,x)$. The function $l(x)$ is defined by $l(x)=[n(n-2)\omega_n\cdot G(p,x)]^{\frac{1}{2-n}}$, where $\omega_n$ is the volume of the unit ball in the $n$ dimensional Euclidean space $\mathbb{R}^n $. Also the asymptotic volume ratio of $M$ is defined as $V_M=\lim\limits_{r\rightarrow \infty}\frac{Vol(B_r(p))}{\omega_nr^n}$, where $B_r(p)$ is the open ball with radius $r$ and center at $p$. For more details see, \cite{GX2019} and also references therein.
\par In this section first we state one theorem and two lemmas from \cite{VB2021} and \cite{GX2019}, which will be used to prove our results:
\begin{thm}\cite{VB2021}\label{sth1}
Let $(M, g, f)$ be a complete non-compact non-steady Schouten soliton such that the potential function $f$ is not-constant. Then for $\lambda >0$ (resp., $\lambda <0$), $f$ attains a global minimum (resp., maximum) and also $f$ is unbounded above (resp., below). Furthermore,
$$0\leq \lambda R\leq 2(n-1)\lambda^2,$$ 
$$2\lambda(f-f_0)\leq|\nabla f|^2\leq 4\lambda(f-f_0),$$
with $f_0=min_{p\in M} f(p)$, if $\lambda>0$ (resp., $f_0=max_{p\in M} f(p)$, if $\lambda <0$).
\end{thm}
\begin{lem}\cite{GX2019}\label{gx1}
If $(M,g)$ is an n-dimensional complete and non-compact Riemannian manifold such that it is not-parabolic with non-negative Ricci curvature, then  $$\lim\limits_{r\rightarrow\infty}\frac{\int_{l\leq r}|\nabla l|^3}{r^n}=(V_M)^{\frac{1}{n-2}}\omega_n.$$
\end{lem}
\begin{lem}\cite{GX2019}\label{gx2}
If $(M,g)$ is an $n(\geq 3)$-dimensional complete Riemannian manifold with non-negative Ricci curvature and the volume growth is maximal (resp., not maximal), then $|\nabla l|\leq 1$ (resp., $\lim\limits_{r\rightarrow \infty}\sup_{t(x)=r}|\nabla l|(x)=0$).
\end{lem}
\begin{thm}\label{th1}
Let $(M,g, f)$ be a complete non-compact gradient shrinking Schouten soliton of dimension $n(> 4)$ with $R\leq k<\frac{(n-1)(n-4)}{n-2}$ for some real constant $k$ and the potential function $f$ is positive non-constant. Then all the ends of $M$ are non-parabolic.
\end{thm}
\begin{proof}
For $a=\frac{n-4}{4}-\frac{k(n-2)}{4(n-1)}>0$, using the equation $(\ref{ss1})$ and the Theorem $\ref{sth1}$ we calculate
\begin{eqnarray}
\nonumber\Delta f^{-a}&=&-af^{-a-1}\Delta f +a(a+1)f^{-a-2}|\nabla f|^2\\
\nonumber&\leq& -a\Big\{\frac{n}{2}+\frac{nR}{2(n-1)}-R\Big\}f^{-a-1}+a(a+1)\{2(f-f_0)\}f^{-a-2}\\
\nonumber&=& \Big\{-a\Big(\frac{n}{2}+\frac{nR}{2(n-1)}-R\Big)+2a(a+1)\Big\}f^{-a-1}-2a(a+1)f_0f^{-a-2}\\
\nonumber&\leq& \Big\{-a\Big(\frac{n}{2}+\frac{nR}{2(n-1)}-R\Big)+2a(a+1)\Big\}f^{-a-1}\\
\nonumber&\leq& a\Big\{\frac{k(n-2)}{2(n-1)}-\frac{n}{2}+2(a+1)\Big\}f^{-a-1}
=0.
\end{eqnarray}
Hence it follows that $f^{-a}$ is a positive superharmonic function which converges to zero at infinity. This proves that (see, \cite{AG1999}) any end of $M$ and hence $M$ is non-parabolic.
\end{proof}
The following corollaries immediately follows from Lemma $\ref{gx1}$, Lemma $\ref{gx2}$ and Theorem $\ref{th1}$:
\begin{cor}
Let $(M, g, f)$ be a complete non-compact gradient shrinking Schouten soliton of dimension $n(> 4)$ with $R\leq k<\frac{(n-1)(n-4)}{n-2}$ for some real constant $k$, $Ric\geq 0$ and the potential function $f$ is non-constant with $min_{p\in M} f(p)=f_0\geq 0$. Then the following relation holds: $$\lim\limits_{r\rightarrow\infty}\frac{\int_{l\leq r}|\nabla l|^3}{r^n}=(V_M)^{\frac{1}{n-2}}\omega_n.$$
\end{cor}
\begin{cor}
Let $(M, g, f)$ be a complete non-compact gradient shrinking Schouten soliton of dimension $n(> 4)$ with not maximal volume growth, $R\leq k<\frac{(n-1)(n-4)}{n-2}$ for some real constant $k$, $Ric\geq 0$ and the potential function $f$ is non-constant with $min_{p\in M} f(p)=f_0\geq 0$. Then $\lim\limits_{r\rightarrow \infty}\sup_{t(x)=r}|\nabla l|(x)=0$. Furthermore, if it has maximal volume growth, then $|\nabla l|\leq 1$.
\end{cor}
\begin{thm}
Let $(M,g,f)$ be a complete non-compact gradient expanding Schouten soliton with $-(n-1)<k_1\leq R$ for some real constant $k_1$ and the potential function $f$ is positive non-constant with $f^{-b}$ bounded above. Then all the ends of $M$ are non-parabolic.
\end{thm}
\begin{proof}
For $b=\frac{n-2}{2}+\frac{k_1(n-2)}{2(n-1)}>0$, using the equation $(\ref{ss1})$ and the Theorem $\ref{sth1}$ we calculate
\begin{eqnarray}
\nonumber\Delta f^{-b}&=&-bf^{-b-1}\Delta f +b(b+1)f^{-b-2}|\nabla f|^2\\
\nonumber&\geq& -b\Big\{-\frac{n}{2}+\frac{nR}{2(n-1)}-R\Big\}f^{-b-1}+b(b+1)\{(f_0-f)\}f^{-b-2}\\
\nonumber&=& \Big\{-b\Big(-\frac{n}{2}-\frac{(n-2)R}{2(n-1)}\Big)-b(b+1)\Big\}f^{-b-1}+b(b+1)f_0f^{-b-2}\\
\nonumber&\geq&\Big\{-b\Big(-\frac{n}{2}-\frac{(n-2)R}{2(n-1)}\Big)-b(b+1)\Big\}f^{-b-1} \\
\nonumber&\geq& b\Big\{\frac{k_1(n-2)}{2(n-1)}+\frac{n}{2}-(b+1)\Big\}f^{-b-1}
=0.
\end{eqnarray}
Hence it follows that $f^{-b}$ is a subharmonic function which is bounded above. This proves that (see, \cite{AG1999}) any end of $M$ and hence $M$ is non-parabolic.
\end{proof}

\section{acknowledgment}
 The second author gratefully acknowledges to the
 CSIR(File No.:09/025(0282)/2019-EMR-I), Govt. of India for financial assistance.

\end{document}